\documentclass{llncs}
\usepackage{amsmath,amssymb}
\usepackage{graphicx}
\usepackage{xspace}
\usepackage{bussproofs}
\usepackage{comment}

\newtheorem{observation}[theorem]{Observation}

\newcommand\Worms{{\mathbb W}}

 \mathchardef\mhyphen="2D

\newcommand{\T}{\ensuremath{\mathrm{T}}\xspace}
\newcommand{\U}{\ensuremath{\mathrm{U}}\xspace}

\newcommand{\RFN}{\ensuremath{{\sf RFN}}\xspace}
\newcommand{\RR}{\ensuremath{{\sf RR}}\xspace}

\newcommand{\glp}{{\ensuremath{\mathsf{GLP}}}\xspace}

\newcommand{\pra}{\ensuremath{{\mathrm{PRA}}}\xspace}

\newcommand{\isig}[1]{{\ensuremath {\mathrm{I}\Sigma_{#1}}}\xspace}

\newcommand{\EA}{\ensuremath{{\rm{EA}}}\xspace}

\newcommand{\la}{\langle}
\newcommand{\ra}{\rangle}
\newcommand{\modal}{{\bf mod}}

\newcommand{\ea}{\ensuremath{{\mathrm{EA}}}\xspace}
\newcommand{\pa}{\ensuremath{{\mathrm{PA}}}\xspace}

\title{The Reduction Property Revisited}
\author{Nika Pona\inst{1} \and Joost J. Joosten\inst{1}}
\institute{University of Barcelona}
\begin{document}
\maketitle
\begin{abstract}
In this paper we will study an important but rather technical result which is called The Reduction Property. The result tells us how much arithmetical conservation there is between two arithmetical theories. Both theories essentially speak about the fundamental principle of reflection: if a sentence is provable then it is true. The first theory is axiomatized using reflection axioms and the second theory uses reflection rules. The Reduction Property tells us that the first theory extends the second but in a conservative way for a large class of formulae. 

We extend the Reduction Property in various directions. Most notably, we shall see how various different kind of reflection axioms and rules can be related to each other.
Further, we  extend the Reduction Property to transfinite reflection principles. Since there is no satisfactory (hyper) arithmetical interpretation around yet, this generalization shall hence be performed in a purely algebraic setting.

For the experts: a consequence of the classical Reduction Property characterizes the $\Pi^0_{n+1}$ consequences and tells us that for any theories $U$ and $T$ of the right complexity we have
\[
U + {\sf Con}_{n+1}(T) \equiv_{\Pi^0_{n+1}} U \cup \{ {\sf Con}_n^k(T)\mid k<\omega\}. 
\]
We will compute which theories can be put at the right-hand side if we are interested in $\Pi^0_j$ formulas with $j{\leq}n$. We answer the question also in a purely algebraic setting where $\Pi^0_j$-conservation will be suitably defined. The algebraic turn allows for generalizations to transfinite consistency notions.
\end{abstract}

\section{Introduction}

G\"odel's celebrated second incompleteness theorem roughly states that any reasonable theory will not prove its own consistency. A theory $T$ is called consistent if no contradiction can be proven. We shall consider theories that contain a minimal amount of arithmetic where consistency is thus equivalent to stating that $0=1$ cannot be proven. 

% In the era of computers that we live in it does not come as a surprise that
We can represent syntactical objects such as proofs by numbers, just as a text file will be represented by a binary number inside a computer. The numbers representing syntax in arithmetic are called G\"odel numbers. Simple operations on syntax like substitution correspond to easy arithmetical operations on the corresponding G\"odel numbers.  In this vein, a theory $T$ with an easy axiom set allows for an arithmetical formula ${\sf Axiom}_T (x)$ that represents this axiom set in the sense that 
\begin{equation}\label{equation:axiomsDefined}
\mathbb N \models {\sf Axiom}_T (x) \ \ \ \Longleftrightarrow \ \ \ x \mbox{ is the G\"odel number of an axiom of $T$}.
\end{equation}
Using this ${\sf Axiom}_T (x)$ formula one can, as G\"odel did in his seminal paper \cite{OnFormalyUndecidableBlaBlaBla}, write a provability predicate $\Box_T$ for the theory $T$ so that
\[
\mathbb N \models \Box_T (x) \ \ \ \Longleftrightarrow \ \ \ x \mbox{ is the G\"odel number of a formula that has a proof in $T$}.
\]
For the sake of readability we shall refrain from differentiating between a syntactical object like a formula $\varphi$ and its G\"odel number as the context should always make clear which is meant where. With this reading convention and a provability predicate as above we can now write G\"odel's second incompleteness theorem succinctly as $T \nvdash \neg \Box_T 0=1$, or, equivalently as: 
\begin{equation}\label{equation:GoedelSecond}
T \nvdash \Box_T 0=1 \to 0=1.
\end{equation}
This formulation of G\"odel's second incompleteness theorem readily suggests a general format and strengthening which is called \emph{reflection}: that what is provable, is actually true. In symbols, $\Box_T \varphi \to \varphi$. 

Philosophically speaking, reflection is an interesting principle. It seems that if one commits to the axioms of a theory $T$, one should also commit to reflection over $T$, yet reflection itself does not follow from the axioms of $T$ as \eqref{equation:GoedelSecond} showed us. As such, and due to the many applications that we shall see, reflection is a fundamental concept in mathematical logic and in the study on the foundations of mathematics.

In this paper, we shall study certain aspects of reflection. In particular, we will study a relation between reflection formulated as an axiom on the one hand and reflection formulated as a rule on the other hand. To formulate the exact statement of our study and provide it with due motivation we first need some definitions and notation. 

Namely, for various reasons it turns out to be natural, important and useful to restrict the formulas $\varphi$ that occur in the reflection principle $\Box_T \varphi \to \varphi$ to certain natural formula classes. Thus we should first say some words on our syntax. For the sake of the paper being self-contained we shall outline the syntactical notions and refer the reader to any standard work (e.g.~\cite{HajekAndPudlak}) on the topic for the details. 

In this paper we will only consider theories in the language of arithmetic, although all results can be generalized to a broader setting. Thus, our language will consist of the usual arithmetic constants $0$ and $1$, operations $+$, $\cdot$ and $2^x$ for addition, multiplication and exponentiation and, the binary relation $\leq$. Terms are defined as usual using these symbols. We call a formula \emph{bounded} if any quantified variable $x$ is bounded by some term $t$ that does not contain $x$. Thus, we only allow quantifiers of the form $\forall\, x{\leq}t$ or $\exists\, x{\leq}t$. The set of all bounded formula is denoted by $\Delta_0$. We inductively define $\Pi_0 := \Sigma_0 := \Delta_0$ and $\Pi_{n+1} := \{ \forall x_0, \ldots, x_m \ \varphi \mid \varphi \in \Sigma_n\}$ and likewise $\Sigma_{n+1} := \{ \exists x_0, \ldots, x_m \ \varphi \mid \varphi \in \Pi_n\}$.

Tarski proved that there is no arithmetical formula $\mathrm{True}(x)$ that is true in the standard model of arithmetic of exactly the G\"odel numbers of formulas that are true in the standard model of arithmetic. However, it is well-known that partial truth predicates do exist. For example, we have a predicate $\mathrm{True}_{\Pi_{n}} (x)$ so that for every formula $\pi$ in $\Pi_n$ we have $\mathbb N \models \mathrm{True}_{\Pi_{n}} (\pi) \ \leftrightarrow \ \pi$. For all of the above formula classes, such a partial truth predicate exists and we shall use them freely throughout the paper. Further good properties of the partial truth-predicates are that the equivalence is actually provable in \ea. Moreover, the complexity of truth predicates are as high as the formula class it speaks about. Thus, for example, $\mathrm{True}_{\Pi_{n}}$ is of complexity $\Pi_n$.

The theory \ea of \emph{Elementary Arithmetic} is given by the defining axioms for the arithmetical symbols together with the induction formulas 
\[
I_\varphi := \varphi (0) \wedge \forall \,  x\ \big(\varphi (x) \to \varphi (x+1)\big) \ \longrightarrow \ \forall \, x \ \varphi (x)
\]
for each bounded formula $\varphi$. An arbitrary theory $T$ shall be called \emph{elementary representable} whenever its axiom-set can be defined by a $\Delta_0$ formula in the sense of \eqref{equation:axiomsDefined}. For example, \ea is elementary representable. 

For a natural number $n$ we denote by $\overline n$ its \emph{numeral} which is a term that evaluates to $n$. It is standard to take $\overline n := 0+\overbrace{1+\ldots+1}^{\mbox{ $n$ times}}$. By $\Box_T \varphi (\dot x)$ we shall denote a predicate with free variable $x$ that for each number $x$ states the provability in $T$ of the formula $\varphi (\overline x)$. We are now ready to formulate the central notion of this paper.

\begin{definition} Given an elementary presentable theory \T, and for $n$ a natural number, the uniform reflection principle $\RFN_{\Sigma_n}(\T)$ is the set of sentences
\[
\forall \vec x \ \big ( \, \square_{\T} \varphi(\dot{\vec x}) \rightarrow \varphi ( \vec x) \, \big )
\]
for all $\varphi(\vec x) \in \Sigma_n.$ 
\end{definition}
The principle $\RFN_{\Pi_n}(\T)$ is defined similarly. For various reasons the partial reflection principles are natural and interesting. For one, it is easy to see that they lead to ever increasing independent principles: stronger and stronger versions of the second incompleteness theorem so to say. 

As a matter of fact, it turns out that each of the partial reflection principles is equivalent to ever stronger consistency notions. To make this precise, let us introduce the following definition where we abbreviate $\neg \Box_T \neg \varphi$ by $\Diamond_T \varphi$. 

\begin{definition}
A theory $\T$ is called $n$-consistent if $\T$ together with all true arithmetical formulas of complexity $\Pi_n$ is consistent, that is, if
$\forall \vec x \ \big ( \mathrm{True}_{\Pi_{n}}(\vec x) \rightarrow \Diamond \mathrm{True}_{\Pi_{n}}(\dot{\vec x})\big ) $ holds. We abbreviate this formula by $\la n \ra_{\T}\top$.
\end{definition}

One can easily see that the partial reflection schema is finitely axiomatizable in terms of the partial truth predicate and hence the following holds:

\begin{proposition}\label{RFNnCon}

$\EA \vdash \T + \RFN_{\Sigma_{n}}(\T + \varphi) \equiv  \T + \la n  \ra_{\T} \varphi$. 
\end{proposition}
\begin{proof}

Using the truth predicate the reflection schema $\RFN_{\Sigma_{n+1}}(\T)$ can be expressed  as the formula $\forall \vec x \big (  \square \mathrm{True}_{\Sigma_{n}}(\dot{\vec x}) \rightarrow \mathrm{True}_{\Sigma_{n}}(\vec x)\big )$, which is just the contraposition of arithmetization of $n$-consistency.  This concludes the proof, since $ \EA \vdash \la n \ra_{\T} \varphi \leftrightarrow \la n  \ra_{\T + \varphi} \top$, by the formalized deduction theorem.
\end{proof}

We further mention that reflection principles are related to various other branches of mathematical logic. For example, it has been proven that the reflection principle $\RFN_{\Sigma_{n+1}}(\ea)$ is (provably over some rather weak theory) equivalent to $\isig{n+1}$. Here, \isig{n+1} is as \ea only that we now allow for induction formulas for any $\Sigma_{n+1}$ formula. The point that we wish to make here is that reflection principles are natural and have various applications.

Note that our formulation of reflection is given in terms of axioms. One can also formulate a rule-based version of reflection. The rule-based version of reflection has better proof-theoretical and computational properties which motivates their usage. 

In formulating the corresponding reflection rule, one has to be a bit careful. For example, it is easy to see that the rule $\frac{\Box_T\varphi}{\varphi}$ for $\varphi$ a \emph{sentence} is simply admissible as a rule for theories that only prove correct $\Sigma_1$ sentences. On the other hand, if we allow parameters, adding to a theory $T$ the rule $\frac{\forall \vec x \square_{\T} \varphi (\dot {\vec x})}{\forall \vec x \varphi(\vec x)}$ turns out to be equivalent to adding to $T$ the full reflection principle $\RFN(\T)$ for all formulas with parameters (see Beklemishev \cite{Beklemishev2005_ReflectionPrinciples}, p. 14 for the proof of this fact).

The better behaved reflection rule turns out to be the one that is obtained by contraposing a certain form of the reflection axiom. This motivates the following definition.

\begin{definition} Given an elementary presentable theory \T and $m, n < \omega$, the $n$-reflection rule over \T is: 
\begin{prooftree}
\AxiomC{$\varphi$}
\LeftLabel{$\RR^n(\T)$:}
\UnaryInfC{ $ \la n \ra_{\T} \varphi$}
\end{prooftree}
Then $\Pi_m\mhyphen\RR^n(\T)$ is the $n$-reflection rule with $\varphi(\vec{x}) \in \Pi_{m}$.
\end{definition}

Although the different versions of this rule look very technical and artificial, various fragments of arithmetic can naturally be described by them. For example, it is known that Primitive Recursive Arithmetic is equivalent to $\ea + \Pi_2\mhyphen\RR^n(\ea)$.
Now that we have given the definition of the reflection rule, we can finally state the theorem called the \emph{Reduction Property} which is the main topic and starting point of this paper. It is easy to see that the reflection axiom is stronger than the corresponding principle. The reduction property tells us that it is not too much stronger.

\begin{theorem}[Reduction property] \label{ReductionFull} Let \T be an elementary presented theory containing \EA, and let
\U be any $\Pi_{n+2}$-extension of \EA. Then $\U +\RFN_{\Sigma_{n+1}}(\T)$ is $\Pi_ {n+1}$-conservative over $\U + \Pi_{n+1}\mhyphen\RR^n(\T)$. 
\end{theorem}

\begin{proof}

See \cite{Beklemishev2005_ReflectionPrinciples}: use cut-elimination to reduce $\RFN_{\Sigma_{n+1}}(\T)$ to $\Pi_ {n+1}\mhyphen\RR^n(\T)$. (Hence this fact is provable in $\EA^+$. Here $\EA^+$ is as \EA together with an axiom that states the totality of the super-exponential function.) 
\end{proof}

The technicality of the theorem draws away the attention from its strength. Let us just mention that for $n=1$ and $T=U = \EA$, the Reduction Property gives us that \isig{1} is $\Pi_2$ conservative over \pra (Parson's Theorem). Furthermore, the Reduction Property is the central ingredient to perform $\Pi_1$-ordinal analysis for \pa and its kin based on provability logics (see \cite{Beklemishev2004}). 

In this paper we will do the following. In Section \ref{section:provLogics} we will revisit some definitions and results from polymodal provability logic. These logics are used in Section \ref{section:RPvariations} where we give prove our first variations on the reduction property. First it is observed that significant simplifications of the reduction property arise when $U=T$. Next, over this simplification we prove various generalizations. Most notably, we shall see how various different kind of reflection axioms and rules can be related to each other. 

In Section \ref{section:algebraicTurn} we shall also extend the Reduction Property to transfinite reflection principles. Since there is no satisfactory (hyper) arithmetical interpretation around yet, this generalization shall hence be performed in a purely algebraic setting.

\section{Polymodal provability logic}\label{section:provLogics}

Many of our results are stated using formalized provability notions like $[n]_T\varphi$. The structural behavior of those predicates is described by what is called \emph{polymodal provability logic}. This is a propositional modal logic with for each given ordinal $\alpha$ a modality $[\alpha]$. 

\begin{definition} 
For $\Lambda$ an ordinal or the class of all ordinals, the logic $\mathsf{GLP}_\Lambda$ is given by the following axioms:
\begin{enumerate}
\item all propositional tautologies{,}
\item Distributivity:
$[\xi](\varphi \to \psi) \to ([\xi]\varphi \to [\xi]\psi)$ for all $\xi<\Lambda${,}

\item Transitivity:
{$[\xi] \varphi \to [\xi] [\xi]\varphi$ for all $\xi<\Lambda$}{,}

\item L\"ob:
{$[\xi]([\xi]\varphi \to \varphi)\to[\xi]\varphi$ for all $\xi<\Lambda$}{,}

\item Negative introspection:
$\la\zeta\ra\varphi\to\la\xi\ra\varphi$ for $\xi<\zeta<\Lambda${,}

\item Monotonicity:
$\la\xi\ra\varphi\to [\zeta]\la\xi\ra\varphi$ for $\xi<\zeta<\Lambda$.
\end{enumerate}
The rules are Modes Ponens and Necessitation for each modality: $\displaystyle \frac{\varphi}{[\xi]\varphi}$.
\end{definition}

The following proposition is often used without explicit mention throughout the paper.

\begin{proposition}\label{PushBiggerDiamondOut}
Let $\Lambda$ be an ordinal, $\gamma > \zeta \in \Lambda$ and $\varphi, \psi$ are $\glp_\Lambda$ formulae, then: \[ \vdash_{\glp_\Lambda} \la \gamma \ra ( \varphi \wedge \la \zeta \ra \psi) \leftrightarrow (\la \gamma \ra \varphi \wedge \la \zeta \ra \psi ) \] 
\end{proposition}

\begin{comment}
\begin{proof}
$(\rightarrow)$ Follows by monotonicity and transitivity and the fact that $ \vdash_{\glp_\Lambda} \la n \ra ( \varphi \wedge \psi ) \rightarrow \la n \ra \varphi \wedge \la n \ra \psi $: 
 \[ \la \gamma \ra ( \varphi \wedge \la \zeta \ra \psi) \Rightarrow \la \gamma \ra \varphi \wedge \la \gamma \ra \la \zeta \ra \psi \Rightarrow \la \gamma \ra \varphi \wedge \la \zeta \ra \la \zeta \ra \psi \Rightarrow \la \gamma \ra \varphi \wedge \la \zeta \ra \psi. \]

$(\leftarrow)$ By axiom $\la \zeta \ra \psi \rightarrow [\gamma] \la \zeta \ra \psi$ we have that $\la \gamma \ra \varphi \wedge \la \zeta \ra \psi \Rightarrow \la  \gamma \ra \varphi \wedge [\gamma] \la \zeta \ra \psi   \Rightarrow \la \gamma \ra (\varphi \wedge \la \zeta \ra \psi ).$
\end{proof}
\end{comment}

At times we shall write $\varphi \vdash_\glp \psi$ instead of $\vdash_\glp \varphi \to \psi$ and shall drop subscripts if the context allows us to. When both $\varphi\vdash \psi$ and $\psi \vdash \varphi$ we will write $\varphi \equiv \psi$. We are interested in a particular subclass of formulae of $\glp_\Lambda$ called worms. They represent iterated consistency statements.

\begin{definition} Given an ordinal $\Lambda$, the set of worms $\Worms^{<\Lambda}$ is inductively defined as follows: \begin{itemize}
  \item $\top \in \Worms$;
    \item $ \la \gamma \ra A \in \Worms^{<\Lambda}$ if $\gamma < \Lambda$ and $A \in \Worms^{<\Lambda}$.
    \end{itemize}
    We write $\Worms^{<\Lambda}_\gamma$ for the collection of worms where all modalities are smaller than $\gamma$. Sometimes we omit $\Lambda$ when the context permits so.
  \end{definition}

For $A$ a worm, we denote by $\alpha\uparrow A$ the worm that arrises by replacing each modality $\la \xi\ra$ by $\la \alpha + \xi\ra$. Likewise, for $A\in \Worms_\alpha$ we denote by $\alpha\downarrow A$ the worm that arrises by replacing each modality $\la \xi\ra$ by $\la -\alpha + \xi\ra$. As usual, for $\alpha \leq \xi$ the result of $-\alpha + \xi$ is defined as the unique ordinal such that $\alpha + (-\alpha + \xi) = \xi$. Sometimes we will omit the modality brackets when writing worms. For example we will write $\zeta A$ instead of $\la \zeta\ra A$.

 Worms in $\Worms_\gamma$ are lineraly ordered (see \cite{Beklemishev2014}) using the following relation:
  \begin{definition}{$(<, <_\gamma)$} $A <_\gamma B \iff B \vdash_{\glp_\Lambda} \la \gamma \ra A.$ For $<_0$ we write $<$.
  \end{definition}

  Worms can be mapped to ordinals with the following isomorphism between $\la \Worms{/}\equiv, <_0\ra$ and $\la {\sf Ord}, <\ra$:

\begin{enumerate}
\item $o(\top) = 0$;

\item $o(A) =  o(b(A)) + \omega^{o(1\downarrow h(A))} + 1$ if $A \not = \top$ and $min \; A = 0$;

\item $o(A) = e^\mu o(\mu \downarrow A)$ if $A \not = \top$ and $\mu = min \; A > 0$.

\end{enumerate}

Here, $e^\mu$ denotes a hyperexponential function defined in \cite{FERNANDEZDUQUE2013785}. Basically, $e^\mu$ is `$\mu$-times iteration' of the function $\alpha \mapsto -1 + \omega^\alpha$ with $e^0$ being the identity. For the purpose of this paper it is not really important to know how this iteration is exactly defined other than $e^{\alpha + \beta} = e^\alpha \circ e^\beta$.

\begin{definition}[$h_\alpha$, $(o_\gamma)$]
For any worm A, we define its $\gamma$-head denoted by $h_\gamma(A)$ as follows: $h_\gamma(\top) := \top$, $h_\gamma(\zeta A) := \zeta h_\gamma(A)$ if $\gamma \leq \zeta$ and  $h_\gamma(\zeta A) := \top$ if $\zeta < \gamma$.  
With this notion of head, we can now define an order function for any ordinal.
\[
o_\gamma (A) = o(\gamma \uparrow h_\gamma(A)).
\]
\end{definition}

The generalized order relation is related to the generalized order function as expected:

\begin{lemma} \label{GammaOrder} For $A,B \subseteq \Worms_\gamma$, $ o_\gamma(A) < o_\gamma(B) \iff A <_\gamma B. $
\end{lemma}
\begin{comment}
\begin{proof} Applying lemmas from Section 4 of \cite{David}.
 $$ o_\gamma(A) < o_\gamma(B) \iff o(\gamma \downarrow A) > o(\gamma \downarrow B) \iff $$
$$ \gamma \downarrow A > \gamma \downarrow B \iff \gamma \uparrow (\gamma \downarrow B ) <_\gamma  \gamma \uparrow (\gamma \downarrow A ) $$
  $$\iff B <_\gamma A $$
\end{proof}
\end{comment}

As we shall see, various aspects of the Reduction Property can be formulated using formalized provability predicates. Since the structural behaviour of those predicates is given by \glp in a sense specified below, we can use all our \glp reasoning inside arithmetical arguments. 

The relevant link between \glp and arithmetic for this paper is the soundness theorem. Key to this theorem is the notion of an \emph{arithmetical $T$-realization} which is a map $\star : {\sf Prop} \to {\sf Form}$ from propositional variables to formulas in the language of arithmetic which is extended to range over all modal formulas of $\glp_\omega$ by stipulating that $\star$ commutes with boolean connectives like $(\varphi \wedge \psi)^*= \varphi^* \wedge \psi^*$ and moreover, $(\la n \ra \varphi)^*:= \la n \ra_T \varphi^*$.

\begin{theorem}[Arithmetical soundness]
Given an elementary presented theory $T$ and $U$ any arithmetical theory extending \ea, and letting $\star$ range over arithmetical $T$-realizations, we have that
\[
\glp_\omega \vdash \varphi \ \ \ \Longrightarrow \ \ \ \forall \star U \vdash \varphi^*.
\]
\end{theorem}

%\input{Prerequisites.tex}

%\section{snippets that may be included}
%
%Note that the reflection rule with parameters is reducible to non-parametric (local) reflection rule. Assume that we have $ \varphi(\vec x )$, then by generalization we get $\forall \vec x \varphi (\vec x)$ and by $\RR^n$ we get $\la n \ra \forall \vec x \varphi (\vec x)$ and $\forall \vec x \la n \ra  \varphi (\vec x)$  by L\"ob's condition and thus  $\la n \ra  \varphi (\vec x)$. This is not the case for the ``contraposition'' of the rule, since $\frac{\square_{\T} \varphi}{\varphi}$ is admissible in $\T$,  Clearly, $\frac{[n]_{\T} \varphi}{\varphi} \to \frac{\square_{\T} \varphi}{\varphi}$, same for the parametric version.
%

\section{Variations of the reduction property}\label{section:RPvariations}

In this section we shall give a genuine generalization of the full Reduction Property. But we will first see that when taking $U =T$ in the reduction property, that this allows for substantial simplifications. It is of this simplification that we shall prove various algebraic generalizations in the next section.

\subsection{A simplification and iterated consistency}
%First we state the reduction property in terms of $\Pi_n$-conservativity of arithmetical theories and prove its simple generalization. Then we move to the algebraic setting.
% For $n > 0$ the scheme $\RFN_{\Sigma_n}(\T)$ is equivalent to $\RFN_{\Pi_{n+1}}(\T)$ and $\la n \ra_{\T} \top$, $n$-consistency of \T. 

It is convenient to axiomatize the Reduction property using Beklemishev's $Q$-formulae.

\begin{definition}[Iterated $n$-consistency]
Let \T be any arithmetic theory containing \EA and $\varphi$ be any formula in the language of \T. Define:
\[ Q_{n}^{0}(\varphi) = \top \]
 \[Q_{n}^{k+1}(\varphi)= \la n \ra_{\T} (\varphi \wedge Q_{n}^{k}(\varphi)).\]
\end{definition}

For $\varphi = \top$ these $Q$ formulas just reduce to iterated consistency.

\begin{observation} $Q_{n}^{k}(\top) \leftrightarrow \la n \ra_{\T}^{k} \top.$
 \end{observation}

The main ingredients of the following proposition were almost formulated as such in \cite{Joosten2013}. The proposition tells us that we can simplify the reduction property significantly in case $T = U$.

%@article{Joosten:2013:AnalysisBeyondFO,
%author={Joosten, J. J.},
%title={{$\Pi^0_1$}-ordinal analysis beyond first-order arithmetic},
%journal={Mathematical Communications},
%volume={18},
%issue={1},
%pages={109-121},
%year={2013},
%}

\begin{proposition}\label{ReductionRuleEquiv}

\begin{minipage}[t]{\linewidth}
\begin{enumerate}

\item $\T + \RR^n(\T + \varphi) \equiv \T + \{ Q^k_n(\varphi) \mid k<\omega \}$ 
\item $\T + \Pi_{n+1}\mhyphen \RR^n(\T + \varphi) \equiv \T + \{ Q^{k}_n (\varphi) \mid k<\omega \}$ 
\item  $ \T + \RR^n(\T + \varphi) \equiv \T +  \Pi_{n+1}\mhyphen \RR^n(\T + \varphi) $.

\end{enumerate}

\end{minipage}

\end{proposition}

\begin{proof}

$(\supseteq, 1\mhyphen2)$ For this direction, Statement 2 implies 1, so it is sufficient to show that for all $k < \omega$, $T + \Pi_{n+1}\mhyphen \RR^n(\T + \varphi) \vdash Q_n^k(\varphi)$. We show this by induction on $k$. Obviously, this implies that $ T + \RR^n(\T + \varphi) \vdash Q_n^k(\varphi)$. For $k = 0$ this is trivial. For the inductive case, assume that $T + \Pi_{n+1}\mhyphen\RR^n(\T + \varphi) \vdash Q_n^k(\varphi)$. Since $Q_n^k(\varphi) \in \Pi_{n+1}$,  by one application of the $\Pi_{n+1}\mhyphen\RR^n(\T + \varphi)$ rule we get $T + \Pi_{n+1}\mhyphen\RR^n(\T + \varphi)  \vdash \la n \ra (\varphi \wedge Q_n^k(\varphi))$.

$(\subseteq, 1\mhyphen2)$ For this direction, Statement 1 implies 2, so it is sufficient to prove the first statement. Assume that $T + \RR^n(\T + \varphi) \vdash \chi$. We show that $\T + \{ Q^k_n(\varphi) \mid k<\omega \} \vdash \chi$ by induction on $l$, the number of applications of the $\RR^n(\T)$ rule. For $l = 0$, we have that $\T \vdash \chi$ and thus $\T + \{ Q^k_n(\varphi) \mid k<\omega \} \vdash \chi$. Now assume that $T + \RR^n(\T + \varphi) \vdash \chi$ by $l + 1$ applications of $\RR^n$ rule and wlog that $\chi = \la n \ra ( \varphi \wedge \chi_1 )$ and  $T + \RR^n(\T + \varphi) \vdash \chi_1$. Then by the induction hypothesis, $\T + Q^k_n(\varphi) \vdash \chi_1$ for some $k$. By the deduction theorem we get $\T  \vdash Q^k_n(\varphi) \to \chi_1$ and by necessitation  $\T  \vdash [n] \big ( Q^k_n(\varphi) \to \chi_1 \big )$. By definition, $\T + \{ Q^k_n(\varphi) \mid k<\omega \} \vdash \la n \ra \big ( \varphi \wedge Q_n^k(\varphi)\big )$ and thus $\T + \{ Q^k_n(\varphi) \mid k<\omega \} \vdash \la n \ra (\varphi \wedge \chi_1)$. 

3. Statement 3 obviously follows from 1 and 2.

   %$\T + \RR^n(\T + \varphi)  \vdash \la n \ra (\varphi \wedge Q^k_n(\varphi) \to \chi_1)$ by one application of $\RR^n$.  By necessitation, $\T + \{ Q^k_n(\varphi) \mid k<\omega \} \vdash [n] Q^k_n(\varphi)$ and thus $\T + \{ Q^k_n(\varphi) \mid k<\omega \} \vdash \la n \ra (\varphi \wedge Q^k_n(\varphi) \wedge Q^k_n(\varphi) \to \chi_1) \vdash \la n \ra (\varphi \wedge \chi_1) = \chi$. This concludes the proof.

\end{proof}

%($\subseteq$) We show that for all $k < \omega$, $T + \Pi_{n+1} \RR^n(\T + \varphi) \vdash Q_n^k(\varphi)$, by induction on $k$. For $k = 0$ trivial. Assume $T + \Pi_{n+1}\RR^n(\T + \varphi)  \vdash Q_n^k(\varphi)$ and since $ Q_n^k(\varphi) \in \Pi_{n+1}$, by one application of the $\Pi_{n+1}$-metareflection rule $T + \Pi_{n+1}-\RR^n(\T + \varphi)  \vdash \la n \ra (\varphi \wedge Q_n^k(\varphi))$.

%($\supseteq$) Assume that for some $k$, $\T + Q^k_n(\varphi) \vdash \psi$. By definition, $\T + Q^{k+1}_{n}(\varphi) \vdash \la n \ra_{\T}( \varphi \wedge Q^k_n(\varphi) )$, which by formalized deduction is equivalent to $\la n \ra_{\T +  Q^k_n(\varphi) } \varphi$. Then by assumption  $\T + Q^{k+1}_{n}(\varphi) \vdash \la n \ra_{\T + Q^k_n(\varphi) + \psi} (\varphi)$ and hence $\T + Q^{k+1}_{n}(\varphi) \vdash \la n \ra_{\T} ( Q^k_n(\varphi) \wedge \varphi \wedge \psi ) \vdash \la n \ra_{\T} ( \varphi \wedge \psi ) $. Since, by the formalized deduction theorem, $\Pi_ {n+1}\mhyphen\RR^n(\T + \varphi)$ is equivalent to the rule 
%\begin{prooftree}
%\AxiomC{$\psi$}
%\UnaryInfC{ $ \la n \ra_{\T} (\varphi \wedge \psi).$}
%\end{prooftree}

%we conclude the proof.

Note that here it is essential that we consider the reflection rule on extensions of $\T$, since the claim doesn't hold in general. Consider for instance the theory $\EA + \la n + 1 \ra_{\T}\top + \RR_n(\EA)$. It is the consequence of the reduction property that:

\begin{proposition}

 $ ( \EA + \la n + 1 \ra_{\EA}\top ) + \RR^n(\EA) \not \equiv (\EA  +  \la n + 1\ra_{\EA}\top) + \Pi_{n+1}\mhyphen \RR^n(\EA) $.
\end{proposition}

The proof is in the appendix.
%\begin{corollary} 
%$\T_{\mathit{n}}^{\omega}   \equiv \T + \{ \la n \ra_{\T}^k \top ; \ k < \omega \} \equiv \T + \RR^n(\T)$.
%\end{corollary}
%
The fact that the reflection schema is equivalent to $n$-consistency together with the Proposition \ref{ReductionRuleEquiv} allows us to reformulate the reduction property.
 
\begin{corollary}[Arithmetic reduction property]\label{theorem:ReductionPropertyArith}
Let \T be a $\Pi_{n+2}$ extension of \EA and $\varphi$ be any arithmetic formula, then:
\[\EA^+ \vdash \T + \la n+1 \ra_{\T} \varphi \ \equiv_{\Pi^0_{n+1}} \ \T + \{ Q^k_n(\varphi) \mid k<\omega \}.\]
\end{corollary}

\subsection{$\Pi_{j+1}$-consequences}

Now we want to characterize exactly the $\Pi_{j+1}$-consequences of $n+1$ consistency statements, for any $j\leq n$. It is enough to consider only worms due to the following fact:

\begin{proposition}\label{deduction} For any r.e. theory $\T$ and any formula $\varphi$, 

\[\la n \ra_{\T + \varphi} \top \equiv_{\Pi_{n+1}} \{ \la n \ra_{\T + \varphi}^k \top \mid k < \omega \} \]
 \center{iff}
\[ \la n \ra_{\T} \varphi \equiv_{\Pi_{n+1}} \{ Q^k_n(\varphi) \mid k < \omega \}. \]

\end{proposition}

\begin{proof}

By the formalized deduction theorem, $\EA \vdash \la n \ra_{\T + \varphi} \top \leftrightarrow  \la n \ra_{\T} \varphi$ and then, by induction on $k$, $\EA \vdash  Q^k_n(\varphi) \leftrightarrow \la n \ra_{\T + \varphi}^k \top$.
%
%\[ \vdash_{EA} \la n \ra_{T + \varphi} \top \leftrightarrow \neg [ n ]_{T + \varphi} \neg \top   \leftrightarrow  \neg [ n ]_{T} \varphi \rightarrow \bot  \leftrightarrow \la n \ra \varphi, \]
%
%by formalized deduction theorem.
%
\end{proof}

We will now formulate a generalization of the Reduction Property. For this we need to consider the following rule which is new.

\begin{definition} Given an elementary presentable theory \T and $m, n < \omega$, the $j$-$n$-reflection rule over \T is: 
\begin{prooftree}
\AxiomC{$\psi$}
\LeftLabel{$\RR^{jn}(\T)$:}
\UnaryInfC{ $ \la j \ra_{\T} \la n \ra_{\T} \psi$}
\end{prooftree}
\end{definition}

We can now prove a generalization of the reduction property theorem and characterize exactly the $\Pi_{j+1}$ consequences of $\la n + 1 \ra_{\T}\top$.

\begin{theorem} \label{JNReduction}Let \T be an elementary presented theory containing \EA, and let
\U be any $\Pi_{n+2}$-extension of \EA. Then $\U +\RFN_{\Sigma_{n+1}}(\T)$ is $\Pi_ {j+1}$-conservative over $\U + \Pi_{j+1}\mhyphen\RR^{jn}(\T)$. 
\end{theorem}
Note that $\T + \RR^{jn}(\T + \varphi) \equiv  \T + \{ \la j \ra_{\T} ( \varphi \wedge Q^k_n(\varphi) ) \mid i, k < \omega \}$, thus we get the characterization of $\Pi_j$-consequences. 

\begin{proposition}\label{PiJArithmetic}
$\T + \RR^{jn}(\T + \varphi) \equiv  \T + \{ \la j \ra_{\T} ( \varphi \wedge Q^k_n(\varphi) ) \mid i, k < \omega \}$
\end{proposition}
\begin{proof}

The proof of this is analogous to the proof of Proposition \ref{ReductionRuleEquiv}, since the rule $\RR^{jn}(\T + \varphi)$ is equivalent to the rule 
\begin{prooftree}
\AxiomC{$\psi$}
\UnaryInfC{ $ \la j \ra_{\T}(\varphi \wedge \la n \ra_{\T}(\varphi \wedge \psi))$.}
\end{prooftree}

For the $(\supseteq)$ direction after applying the $j\mhyphen$-reflection rule to the I.H. use the fact that $\la j \ra_{\T + \varphi} \la n \ra_{\T + \varphi} \la j \ra Q_n^k(\varphi)$ implies $\la j \ra_{\T + \varphi} \la n \ra_{\T + \varphi} Q_n^k(\varphi)$ for $j \leq n$.

For the $(\subseteq)$ direction one has to apply necessitation and MP under the box twice, for $j$ and $n$, and use the fact that for some $l$, $$ Q_n^{l}(\varphi) \vdash \la j \ra_{\T + \varphi} \la n \ra_{\T + \varphi} \la j \ra_{\T + \varphi} \la n \ra_{\T + \varphi} Q^k_n(\varphi).$$
\end{proof}

\begin{corollary}
Let \T be a $\Pi_{n+2}$ extension of \EA and $\varphi$ be any arithmetic formula, then:
\[\T + \la n+1 \ra_{\T} \varphi \ \equiv_{\Pi^0_{j+1}} \ \T + \{ \la j \ra_{\T}( \varphi \wedge Q^k_n(\varphi)) \mid i,k<\omega \}.\]%\footnote{I decided to go with this formulation instead of $\{\la j \ra ( \varphi \wedge Q^k_n(\varphi)) \mid i,k<\omega \}$ as I did before. Makes proofs simpler and conceptually makes more sense, since the set in the  corollary is weaker. }
\end{corollary}

And as special case just in terms of worms we obviously have:

\begin{corollary}
Let \T be a $\Pi_{n+2}$ extension of \EA and $\varphi$ be any arithmetic formula, then:
\[\T + \la n+1 \ra_{\T} \top \ \equiv_{\Pi^0_{j+1}} \ \T + \{ \la j \ra_{\T} \la n \ra_{\T}^k \top\mid k<\omega \}.\]
\end{corollary}

%By Proposition \ref{deduction}, we have: 
%\begin{corollary}
%Let \T be a $\Pi_{n+2}$ extension of \EA and $\varphi$ be any arithmetic formula, then:
%\[\T + \la n+1 \ra_{\T + \varphi} \top \ \equiv_{\Pi^0_{j+1}} \ \T + \{\la j \ra_{\T +  \varphi} \la n \ra_{\T +  \varphi}^k \mid k<\omega \}.\]
%\end{corollary}

\section{An algebraic formulation}\label{section:algebraicTurn}

%\begin{corollary}[Reduction property for worms]

%If \T is a $\Pi_{n+2}$-extension of \EA, then  $\T+ \la n+1 \ra_{\T} \top$ is a
%$\Pi_{n+1}$-conservative extension of $\T + \{\la  n \ra^k_{\T} \top \mid \ k < \omega \}$.

% $\T^n_{\omega} := \T + n\mhyphen Con(\T) + n\mhyphen Con(\T + n\mhyphen Con(\T)) \ldots \equiv_{\EA} 

%\end{corollary}

The reduction property theorem is formulated in the language of arithmetic. By the Friedman-Goldfarb-Harrington Theorem, for any $\Sigma_1$-formula $\varphi(\vec{x})$ one can prove (in \EA) that if \T is consistent, this formula is equivalent to a formula $\square_{\T}\psi(\dot{\vec x})$ for some $\psi(\vec x)$ 

\begin{theorem}[Friedman-Goldfarb-Harrington] Let \T be a recursively enumerable arithmetic theory\footnote{This implies that it is elementary presentable, by Craig's trick.} and $\varphi(\vec{x}) \in \Sigma_1$, then there is a $\psi(\vec x)$ such that:
$$ \EA \vdash \Diamond_{\T}\top \to \big (\square_{\T}\psi(\dot{\vec x}) \leftrightarrow \varphi(\vec{x}) \big)$$
\end{theorem}

In \cite{Joosten:2015:TuringJumpsThroughProvability} this theorem is generalized for $n$-provability, thus in a sense the arithmetical $\Pi^0_{n+1}$ sentences are entirely captured by sentences of the form $\la n\ra_{\T} \varphi$. Then  it makes sense to ask if the reformulation of Theorem \ref{theorem:ReductionPropertyArith} holds in a purely modal/algebraic setting. This question was answered in the affirmative in \cite{Beklemishev2016}, where Beklemishev proves that the Reduction Property hods in an algebraic setting.

\begin{definition}[Algebraic $n$-conservativity] Let $\tau, \sigma$ be sets of $\glp$ formulae. Then we say that $\sigma$ and $\tau$ are $n$-conservative (we write $\tau \equiv_n \sigma$) if, for all formulae $\varphi$ $$\tau \vdash_{\glp} \la n \ra \varphi \text{ iff } \sigma \vdash_{\glp} \la n \ra \varphi.$$
\end{definition}

With this notion of conservativity we can now formulate an algebraic pendant of the reduction property.

\begin{theorem}[Algebraic reduction property]\label{theorem:ReductionPropertyGLP}
Let $\varphi, \psi$ be any modal formulae, then:
%\[ \la n+1 \ra \varphi \vdash_{\glp} \la n \ra \psi \text{ iff } \{ Q^k_n(\varphi) ; k< \omega \} \vdash_{\glp} \la n \ra \psi.\]
%We abbreviate the above in the following way:
$$   \la n+1 \ra \varphi  \equiv_n \{ Q^k_n(\varphi) ; k< \omega \}. $$
\end{theorem}

\subsection{$\Pi_{j+1}$-consequences: algebraic version}

First, we prove the algebraic version of Proposition \ref{PiJArithmetic}, that is, we characterize exactly the $\Pi_{j+1}$-consequences of $\la n+1 \ra \varphi$ for $j \leq n$.

\begin{observation} \label{PiJ} If $A \equiv_n B$ and $A \vdash \la n \ra \top$, then $ A \equiv_j B $, for $j \leq n$. 
\end{observation}
\begin{proof}
  Assume $A \vdash \la j \ra \varphi$. Since $A \equiv_n B$ and $A \vdash \la n \ra \top$, we have that $B \vdash \la n \ra \top$ and thus $B \vdash \la j \ra \varphi$.

\end{proof}

\begin{theorem}[$\Pi_j$-consequences]\label{PijAlgebraic} For $j \leq n < \omega$,
   $$ \la n + 1 \ra \varphi \equiv_j   \{  \la j \ra (\varphi \wedge Q_n^k(\varphi)); k < \omega \}$$
%  $$ \la n + 1 \ra \varphi \equiv_j   \{  Q(\la j \ra \la n \ra^k\top, \varphi) \mid k < \omega \}$$
  
\end{theorem}

\begin{proof}
By induction on $n - j$. \\

{\em Base case:} $j = n$. By the algebraic version of the Reduction Property (Theorem \ref{theorem:ReductionPropertyGLP}), we have that $$  \la n + 1 \ra \varphi \equiv_n  \{ Q^k_n(\varphi) ; k<\omega \} = \{  \la n \ra (\varphi \wedge Q_n^k(\varphi)); k < \omega \}.$$ 

{\em Inductive Step: }  Now we assume that
$ \la n + 1 \ra \varphi  \equiv_j \{  \la j + 1 \ra (\varphi \wedge Q_n^k(\varphi)); k < \omega \}$ ($\sf{I.H.}$) and want to prove: 
 $ \la n + 1 \ra \varphi \equiv_j\{  \la j  \ra (\varphi \wedge Q_n^k(\varphi)); k < \omega \}.$

For any $k$, by Reduction Property for $j < n$ we have:  
 \begin{equation} \label{11} \la j + 1 \ra (\varphi \wedge Q_n^k(\varphi) ) \equiv_j  \{ Q_j^i ( \varphi \wedge Q^k_n (\varphi) ) ; i < \omega \}.
\end{equation}\
This implies that \begin{equation}\label{add} \{  \la j +1 \ra (\varphi \wedge Q_n^k(\varphi)); k < \omega \} \equiv_{j}  \{ Q_j^i ( \varphi \wedge Q^k_n (\varphi) ) ; i, k < \omega \}. \end{equation}
To prove the main claim we assume that $\la n + 1 \ra \varphi \vdash_{\glp_\omega} \la j \ra \psi$. By $\sf{I.H.}$ and Observation \ref{PiJ} we get 
\begin{equation}\label{22}
 \{  \la j + 1 \ra (\varphi \wedge Q_n^k(\varphi)); k < \omega \} \vdash_{\glp_\omega} \la j \ra \varphi.
\end{equation}
By (\ref{add}) from (\ref{22}) we obtain 
 \begin{equation}\label{44} \{ Q_j^i ( \varphi \wedge Q^k_n (\varphi) ) ; i, k < \omega \} \vdash_{\glp_\omega} \la j \ra \varphi. \end{equation}
By an easy induction we can see that $ \la j \ra (\varphi \wedge Q_n^{i+k}(\varphi)) \vdash_{\glp_\omega}  Q_j^i ( \varphi \wedge Q_n^k (\varphi) ) $ for any $i$ and $k$. Thus we have that $\{  \la j \ra (\varphi \wedge Q_n^k(\varphi)); k < \omega \} \vdash_{\glp_\omega}  \{ Q_j^i ( \varphi \wedge Q_n^k (\varphi) ) ; i, k < \omega \} $ and obtain the conclusion:
$\{  \la j \ra (\varphi \wedge Q_n^k(\varphi)); k < \omega \}  \vdash_{\glp_\omega} \la j \ra \varphi.$

\end{proof}

\begin{corollary} \label{PijConsequences} For $0 \leq j \leq n$, $\{ \la j \ra \la n \ra^{k} \top \mid k < \omega \}  \equiv_j\la n + 1 \ra \top $

\end{corollary}

\subsection{Reduction property in transfinite setting}

When generalizing reduction property to transfinite setting, it is useful to define a more general version of the $Q$-formulae.

\begin{definition} Let $\varphi$ by any formula, then by induction on the length of the worm define:
  $$ Q(\top, \varphi) = \top$$
  $$ Q(\la \gamma \ra A, \varphi) = \la \gamma \ra ( \varphi \wedge Q(A,\varphi))$$
\end{definition}

The next lemma tells us that many important properties of these generalized $Q$-formulae actually resides in the worms involved.

\begin{lemma}\label{Qformulae} Let $A, B \in \Worms$ and $\varphi$ a $\glp_{\Lambda}$ formula, then
  $$ A \vdash B \Rightarrow Q(A,\varphi) \vdash Q(B,\varphi) $$

\end{lemma}

\begin{proof}
  For any \glp formulas $\psi$ and $\chi$, we denote by $\psi^\chi$ the formula that is obtained by replacing each subformula from $\psi$ of the form $[\xi] \psi'$ by $[\xi] (\chi \rightarrow \psi')$. By induction on the length of a \glp proof it is easy to see that if $\psi$ is provable, then so is $\psi^\chi$. Then the lemma follows, since $Q(A,\varphi) = A^\varphi$.
\end{proof}

%
%\begin{definition}[Order type] For $\gamma$ ordinal, $o_\gamma : \Worms \rightarrow On$ is the order type function corresponding to the usual $<_\gamma$-ordering on $\Worms$ (on $\Worms_\gamma$ it is linear). If $\tau \subseteq \Worms_\gamma$ then define $o_\gamma(\tau) = \sup \{ o_\gamma(A) ; A \in \tau \}$. We write $o$ and $<$ for $o_0$ and $<_0$.
%\end{definition}

Since for limit ordinals we cannot talk of a predecessor, we need something similar to that as captured in the notion of cofinality.

\begin{definition}[Cofinal set of worms] Let $\tau \subseteq \Worms_\gamma$ and $A \in \Worms_\gamma$. Then we say that $\tau$ is $<_\gamma$-cofinal in $A$ if $o_\gamma(\tau) = o_\gamma(A)$.

\end{definition}

\begin{comment}
\begin{proposition} If $\tau \subseteq \Worms_\gamma$ is $<_\gamma$-cofinal in $\la \zeta \ra \top$ with $\gamma < \zeta$, then $\tau$ is $<_0$-cofinal in $\la \zeta \ra \top$.
\end{proposition}

\begin{proof}
Since on $\Worms_\gamma$, $<_\gamma$ coincides with $<_0$.
\end{proof}
\end{comment}

The following definition allows us to speak of the set of ordinals that occur in a formula.

\begin{definition} If $\varphi$ is a formula of $\glp_\Lambda$ then $\modal \  \varphi$ is the set of all ordinals appearing in the modalities of $\varphi$. 
\end{definition}

In the following lemma, we recall that it is not the size of the ordinals that matter, rather it is just their comparison to the other ordinals in the formula that matters.

\begin{lemma}[Demotion and promotion lemma] Given $\varphi, \psi$ of $\glp_\Lambda$, let $s : \modal \{ \varphi \wedge \psi \} \rightarrow \omega $ be  a function that enumerates modalities of $\varphi, \psi$ in the increasing order and $\overline \psi$ the result of replacing all modalities $\la \zeta \ra$ in $\psi$ by $\la s(\zeta) \ra$.
  $$ \vdash_{\glp_\Lambda} \varphi \rightarrow \psi \Rightarrow \vdash_{\glp_\omega} \overline \varphi  \rightarrow \overline \psi \; (\sf{Demotion}) $$ 
$$ \vdash_{\glp_\omega} \overline \varphi  \rightarrow \overline \psi 
\vdash_{\glp_\Lambda} \varphi \rightarrow \psi \; (\sf{Promotion}) $$

\end{lemma}
\begin{proof} See \cite{Beklemishev2014}.
\end{proof}

We are now ready to state and prove our transfinite generalization of the Reduction Property.

\begin{theorem} \label{GeneralizedReductionProperty}

Let $\gamma < \zeta \in On$ and $\varphi$ be any modal formula. If $\la \gamma \ra \tau \subseteq \Worms^{<\zeta}$ is $<_\gamma$-cofinal in $\la \zeta \ra \top$,  then
\[
\la \zeta \ra \psi \equiv_{\gamma} Q(\la \gamma \ra \tau,\psi).
\]
\end{theorem}

\begin{proof}

Since all the modalities in $\gamma \tau$ are smaller than $\zeta$, we observe that $\la \zeta \ra \top \vdash_{\glp_\Lambda} \la \gamma \ra \tau$ (that is, for each formula $\chi$ from $\la \gamma \ra \tau$ we have that $\la \zeta \ra \top \to \chi$ is provable). Consequently, by Lemma \ref{Qformulae} we obtain that $\la \zeta \ra \psi \supseteq Q(\la \gamma \ra \tau,\psi)$ whence in particular
$\la \zeta \ra \psi \supseteq_\gamma Q(\la \gamma \ra \tau,\psi)$

So, let us now focus on the reverse inclusion. We thus need to show that for arbitrary $\varphi$ we have $
\la \zeta \ra \psi \vdash_{\glp_\Lambda} \la \gamma \ra \varphi \Rightarrow  Q(\la \gamma \ra \tau,\psi) \vdash_{\glp_\Lambda} \la \gamma \ra \varphi .
$

We assume that $\la \zeta \ra \psi \vdash_{\glp_\Lambda} \la \gamma \ra \varphi$. 
By the Demotion Lemma we get 

\begin{equation}\label{111}  
\la n+1 \ra \overline \psi \vdash_{\glp_\omega} \la j \ra \overline \varphi 
\end{equation} 

for some $j \leq n < \omega$. Proposition \ref{PijConsequences} ($\Pi_j$-consequences) and (\ref{111}) imply 

\begin{equation}\label{222} 
\{  \la j\ra (\psi \wedge Q^k_n(\overline \psi)) \mid k < \omega \} \top \vdash_{\glp_\omega} \la j \ra \overline \varphi.
\end{equation} 

From (\ref{222}), by the Promotion Lemma we get 
\begin{equation}\label{333} 
\{  \la \gamma \ra (\psi \wedge Q^k_\theta(\overline \psi)) \mid k < \omega \}  \top \vdash_{\glp_\Lambda} \la \gamma \ra \varphi \end{equation} 
with $\theta \in \modal  ( \la \gamma \ra \varphi \wedge \psi)$ being the biggest ordinal in $ \la \gamma \ra \varphi \wedge \psi$ smaller than $\zeta$.  Then we are done, once we show that $Q(\la \gamma \ra \tau,\psi) \vdash_{\glp_\Lambda}  Q\big (\la \gamma \ra \la \theta \ra^k  \top, \psi \big ) $ for all natural numbers $k$. By Lemma \ref{Qformulae}, we just need to show that $\la \gamma \ra \tau \vdash \la \gamma \ra \la \theta \ra^k  \top $.

 Since $\theta, \gamma < \zeta$, we have that for each $k < \omega$, $ \la \gamma \ra \la \theta \ra^k \top <_\gamma \la \zeta \ra \top $. Then, by the $<_\gamma$-cofinality of $\la \gamma \ra \tau$ in $\la \zeta\ra \top$ we get that $\la \gamma \ra \tau  \vdash_{\glp_\Lambda} \la \gamma \ra \la \theta \ra^k \top$. Therefore, $Q(\la \gamma \ra \tau,\psi) \vdash \la \gamma \ra \varphi$ and this concludes the proof. 

\end{proof}

%$$ Q_{j,n}^{0}(\varphi) = \top$$

%$$ Q_{j,n}^{k+1}(\varphi)= \la j \ra \la n \ra (\varphi \wedge Q_{j,n}^{k}(\varphi)).$$

\section{Appendix}

{\bf Theorem \ref{JNReduction}.} Let \T be an elementary presented theory containing \EA, and let
\U be any $\Pi_{n+2}$-extension of \EA. Then $\U +\RFN_{\Sigma_{n+1}}(\T)$ is $\Pi_ {j+1}$-conservative over $\U + \Pi_{j+1}\mhyphen\RR^{jn}(\T)$. 

\begin{proof}
  By adapting the proof of Theorem \ref{ReductionFull} from \cite{Beklemishev2005_ReflectionPrinciples}. Consider a cut-free proof of a sequent $\Gamma$ of the form: \begin{equation} \label{CFP} \neg U, \neg \RFN_{\Sigma_{n+1}}(\T), \Pi \end{equation}  where $\neg U$ and $\neg \RFN_{\Sigma_{n+1}}(\T)$ are finite subsets of negations of axioms of $\U$ and of refletion principles, and $\Pi \subseteq\Pi_{j+1}$.  Let $\Gamma^-$ be the result of deleting  from $\Gamma$ all subformulae of $\neg U$ and the subformulae of $\neg \RFN_{\Sigma_{n+1}}(\T)$ of complexity bigger than $\Pi_{n+1}$.

  We show that if there is a cut-free proof of $\Gamma$, then  $\U + \Pi_{j+1}\mhyphen\RR^{jn}(\T) \vdash \bigvee \Gamma^-$. The proof is by induction on the height of a derivation of $\Gamma$.  Here we describe the main step in which the proof differs from our case, for more details see \cite{Beklemishev2005_ReflectionPrinciples}. Assume that the derivation is of the form
  \begin{prooftree}
    \AxiomC{${\sf Prf}_{\T}(t, \ulcorner \neg \varphi (\dot s ) \urcorner ), \Delta$}
    \AxiomC{$\varphi(s), \Delta$}
\RightLabel{$\wedge$-I}
\BinaryInfC{${\sf Prf}_{\T}(t, \ulcorner \neg \varphi (\dot s ) \urcorner )\wedge \varphi(s), \Delta$}
\end{prooftree}
%\begin{prooftree}
%  \AxiomC{$\bigvee \varphi(s) \vee \Delta^-$}
 % \LeftLabel{$\RR^{jn}$}
 
 %\UnaryInfC{$\la j \ra \la n \ra \big ( \bigvee \varphi(s) \vee \Delta^-\big )$}
%\end{prooftree}

By induction hypothesis, we have a derivation in $\U + \Pi_{j+1}\mhyphen\RR^{jn}(\T)$ of \begin{equation}\label{ih1} \varphi(s) \vee \bigvee \Delta^-\end{equation} \center and \begin{equation}\label{ih2}{\sf Prf}_{\T}(t, \ulcorner \neg \varphi (\dot s ) \urcorner )\vee \bigvee\Delta^-\end{equation} Then from (\ref{ih1}) by $\Pi_{j+1}\mhyphen\RR^{jn}$ we get $\la j \ra_{\T} \la n \ra_{\T} \big ( \varphi(\dot s) \vee  \bigvee \Delta^-\big )$ and by L\"{o}b's conditions and provable $\Sigma_{j+1}$-completeness we derive \begin{equation}\label{6}\la j \ra_{\T} \la n \ra_{\T} \varphi(\dot s) \vee \bigvee \Delta^-\end{equation} From (\ref{ih2}) we get $\square_{\T} \neg \varphi(\dot s)\vee \bigvee\Delta^-$ by $\exists$-introduction, which implies \begin{equation}\label{7}[j]_{\T}[n]_{\T} \neg \varphi(\dot s)\vee \bigvee\Delta^-\end{equation} Then by cut from (\ref{6}) and (\ref{7}) we get $\bigvee \Delta^-$, which concludes the proof of this case.
\end{proof}

{\bf Lemma \ref{PushBiggerDiamondOut}.}
Let $\Lambda$ be an ordinal, $\gamma > \zeta \in \Lambda$ and $\varphi, \psi$ are $\glp_\Lambda$ formulae, then:
\[ \vdash_{\glp_\Lambda} \la \gamma \ra ( \varphi \wedge \la \zeta \ra \psi) \leftrightarrow (\la \gamma \ra \varphi \wedge \la \zeta \ra \psi ) \] 

\begin{proof}
$(\rightarrow)$ Follows by monotonicity and transitivity and the fact that $ \vdash_{\glp_\Lambda} \la n \ra ( \varphi \wedge \psi ) \rightarrow \la n \ra \varphi \wedge \la n \ra \psi $: 
 \[ \la \gamma \ra ( \varphi \wedge \la \zeta \ra \psi) \Rightarrow \la \gamma \ra \varphi \wedge \la \gamma \ra \la \zeta \ra \psi \Rightarrow \la \gamma \ra \varphi \wedge \la \zeta \ra \la \zeta \ra \psi \Rightarrow \la \gamma \ra \varphi \wedge \la \zeta \ra \psi. \]

$(\leftarrow)$ By axiom $\la \zeta \ra \psi \rightarrow [\gamma] \la \zeta \ra \psi$ we have that $\la \gamma \ra \varphi \wedge \la \zeta \ra \psi \Rightarrow \la  \gamma \ra \varphi \wedge [\gamma] \la \zeta \ra \psi   \Rightarrow \la \gamma \ra (\varphi \wedge \la \zeta \ra \psi ).$
\end{proof}

{\bf Lemma \ref{GammaOrder}.}  \label{GammaOrder} For $A,B \subseteq \Worms_\gamma$, $ o_\gamma(A) < o_\gamma(B) \iff A <_\gamma B. $

\begin{proof} Applying lemmas from Section 4 of \cite{David}.
 $$ o_\gamma(A) < o_\gamma(B) \iff o(\gamma \downarrow A) > o(\gamma \downarrow B) \iff $$
$$ \gamma \downarrow A > \gamma \downarrow B \iff \gamma \uparrow (\gamma \downarrow B ) <_\gamma  \gamma \uparrow (\gamma \downarrow A ) $$
  $$\iff B <_\gamma A $$
\end{proof}

\begin{proposition}

 $ ( \EA + \la n + 1 \ra_{\EA}\top ) + \RR^n(\EA) \not \equiv (\EA  +  \la n + 1\ra_{\EA}\top) + \Pi_{n+1}\mhyphen \RR^n(\EA) $.
\end{proposition}

\begin{proof}

 Assume that $$ ( \EA + \la n + 1 \ra_{\EA}\top ) + \RR^n(\EA)  \equiv (\EA  +  \la n + 1 \ra_{\EA}\top) + \Pi_{n+1}\mhyphen \RR^n(\EA) .$$ By the reduction property we have that  $$ ( \EA + \la n + 1\ra_{\EA}\top ) + \RR^n(\EA)  \equiv_{\Pi_{n+1}}  ( \EA + \la n + 1 \ra_{\EA}\top ) + \RFN_{\Sigma_{n+1}}(\EA)$$ and by Proposition \ref{RFNnCon} $$ ( \EA + \la n + 1\ra_{\EA}\top ) + \RR^n(\EA)  \equiv_{\Pi_{n+1}}   ( \EA + \la n + 1 \ra_{\EA}\top ) + \la n + 1\ra_{\EA}\top$$ Hence we have that $ ( \EA + \la n + 1\ra_{\EA}\top ) + \RR^n(\EA)  \equiv_{\Pi_{n+1}}   ( \EA + \la n + 1 \ra_{\EA}\top )$. But now, since  $$(\EA + \la n + 1\ra_{\EA}\top ) + \RR^n(\EA)  \vdash \la n \ra_{\EA} \la n + 1\ra_{\EA}\top$$ we arrive at a contradiction:  $$\EA + \la n + 1\ra_{\EA}\top  \vdash   \la n \ra_{\EA +\la n + 1\ra_{\EA}\top} \top.$$

\end{proof}

\bibliography{bib_worms}
\bibliographystyle{splncs04}

\end{document}